\documentclass{amsart}
\usepackage{graphicx,amscd}
\newtheorem{theorem}{Theorem}[section]
\newtheorem{cor}[theorem]{Corollary}
\newtheorem{ex}[theorem]{Exercise}

\begin{document}

\title[Uniform Space Duality]
{A Duality Between Non-Archimedean Uniform Spaces and Subdirect Powers of Full Clones}
\author{Joseph Van Name}
\address{Department of Mathematics and Statistics\\
University of South Florida\\4202 E. Fowler Avenue Tampa, FL 33620\\USA}
\email{jvanname@mail.usf.edu}


\begin{abstract}
A uniform space is said to be non-Archimedean if it is generated by
equivalence relations. If $\lambda$ is a cardinal, then a non-Archimedean
uniform space $(X,\mathcal{U})$ is $\lambda$-totally bounded if each
equivalence relation in $\mathcal{U}$ partitions $X$ into less than $\lambda$
blocks. If $A$ is an infinite set, then let $\Omega(A)$ be the algebra with
universe $A$ and where each $a\in A$ is a fundamental constant and every finitary function
is a fundamental operation. We shall give a duality between
complete non-Archimedean $|A|^{+}$-totally bounded uniform spaces and
subdirect powers of $\Omega(A)$. We shall apply this duality to characterize the algebras
dual to supercomplete non-Archimedean uniform spaces. 
\end{abstract}
\maketitle
\section{Non-Archimedean Uniform Space Duality}

In this paper, we shall assume basic facts about uniform spaces and universal algebra. The reader is referred to \cite{I} or \cite{J}
for information about uniform spaces and to \cite{B} for universal algebra. We shall use the entourage definition of uniform spaces, and we shall
assume all complete uniform spaces are separated. If $\mathcal{A}$ is an algebra, then we shall write
$V(\mathcal{A})$ for the variety generated by $\mathcal{A}$.

In \cite{S}, Marshall Stone constructed a duality between compact totally
disconnected spaces and Boolean algebras. This result revolutionized the theory of Boolean
algebras since it gives a way to represent Boolean algebras as topological spaces. We shall give an
analogous result for uniform spaces.

A uniform space $(X,\mathcal{U})$ is said to be \emph{non-Archimedean} if $\mathcal{U}$ is generated
by equivalence relations. We say that a non-Archimedean uniform space $(X,\mathcal{U})$ is $\lambda$-totally bounded if whenever
$E\in\mathcal{U}$ is an equivalence relation, then $E$ partitions $X$ into less than
$\lambda$ blocks. Clearly, if $(X,\mathcal{U})$ is $\lambda$-totally bounded, then each
subspace of $X$ is $\lambda$-totally bounded as well.

For this paper, let $A$ be a fixed infinite set. For each $a\in A$, let $\hat{a}$ be a constant symbol.
For each $f\colon A^{n}\to A$, let $\hat{f}$ be an $n$-ary function symbol. Let
$\mathcal{F}=\{\hat{a}|a\in A\}\cup\bigcup_{n}\{\hat{f}|f\colon A^{n}\to A\}$. Let
$\Omega(A)$ be the algebra of type $\mathcal{F}$ and with universe $A$ where
$\hat{a}^{\Omega(A)}=a$ for $a\in A$ and where $\hat{f}^{\Omega(A)}=f$ for $f\colon A^{n}\to A$.
Therefore every $n$-ary function on $A$ is given by a function symbol, so we can regard $\Omega(A)$
as the full clone of $A$. We shall now give a duality between subdirect powers of
$\Omega(A)$ and complete non-Archimedean $|A|^{+}$-totally bounded uniform spaces. With this duality,
every complete non-Archimedean uniform space can be represented algebraically simply by letting $|A|$
be at least as large as every uniform partition. 

The algebra $\Omega(A)$ and the variety $V(\Omega(A))$ generated by $\Omega(A)$ have applications
to mathematics besides uniform space duality. For instance,
the variety $V(\Omega(A))$ is related to the ultrapower construction and reduced power construction. In fact, one
can construct ultrapowers and reduced powers from elements of the variety $V(\Omega(A))$.
Also, the first order theory $\textrm{Th}(\Omega(A))$ of $\Omega(A)$ is appealing since it is generated by the identities in $\Omega(A)$ and
a single sentence. In other words, there is a $\phi\in\textrm{Th}(\Omega(A))$ such that for each $\theta\in\textrm{Th}(\Omega(A))$, there are identities
$I_{1},\dots,I_{n}\in\textrm{Th}(\Omega(A))$ such that $(\phi\wedge I_{1}\wedge\dots\wedge I_{n})\rightarrow\theta$.

The algebra $\Omega(A)$ serves as an infinite analogue of the two element Boolean
algebra $B$ since in $B$ every function can be represented as a combination of the Boolean operations $\wedge,\vee,'$.
Therefore the variety $V(\Omega(A))$ is analogous to the variety of Boolean algebras.
The category of compact totally disconnected spaces is isomorphic
to the category of complete non-Archimedean $\aleph_{0}$-totally bounded uniform spaces. 
Therefore it should be possible to reconstruct a duality between compact totally disconnected spaces and
the variety of Boolean algebras, but for simplicity we shall only consider the variety $V(\Omega(A))$
when $A$ is infinite.

An algebra $\mathcal{L}\in V(\Omega(A))$ shall be called \emph{partitionable} if there
is an injective homomorphism $\phi\colon\mathcal{L}\to\Omega(A)^{I}$ for some
set $I$. Clearly, the products and subspaces of partitionable algebras are partitionable.
Furthermore, each partitionable algebra is isomorphic to a subdirect product of $\Omega(A)$ since
each $a\in A$ is a constant in $\Omega(A)$.

Let $Z(\mathcal{L})$ be the collection of all homomorphisms $\phi\colon\mathcal{L}\to\Omega(A)$. 
In this paper, the set $A$ will always have the discrete uniformity. Now give $A^{\mathcal{L}}$ the product uniformity.
Then the topology on $A$ is the discrete topology and the
topology on $A^{\mathcal{L}}$ is the product topology. Give $Z(\mathcal{L})\subseteq A^{\mathcal{L}}$
the subspace uniformity. Then $Z(\mathcal{L})$ is a closed subspace
of $A^{\mathcal{L}}$ since every convergent net $(\phi_{d})_{d\in D}$ in $Z(\mathcal{L})$ converges to some
$\phi\in Z(\mathcal{L})$. Thus, since $Z(\mathcal{L})$ is a closed subspace of a complete uniform space,
$Z(\mathcal{L})$ is complete.

Let $\ell_{1},\dots,\ell_{n}\in\mathcal{L}$.
Then let $\mathcal{E}^{\sharp}_{\ell_{1},\dots,\ell_{n}}$ be the equivalence relation $A^{\mathcal{L}}$
where for $r,s\in A^{\mathcal{L}}$ we have
$(r,s)\in\mathcal{E}^{\sharp}_{\ell_{1},\dots,\ell_{n}}$ if and only if 
$r(\ell_{1})=s(\ell_{1}),\dots,r(\ell_{n})=s(\ell_{n})$. Then the equivalence
relations $\mathcal{E}^{\sharp}_{\ell_{1},\dots,\ell_{n}}$ generate the uniformity on
$A^{\mathcal{L}}$. Take note that each $\mathcal{E}^{\sharp}_{\ell_{1},\dots,\ell_{n}}$ partitions
$A^{\mathcal{L}}$ into $|A|^{n}=|A|$ blocks, so the uniform space $A^{\mathcal{L}}$ is $|A|^{+}$-totally bounded.
Let $\mathcal{E}_{\ell_{1},\dots,\ell_{n}}$ be the restriction of
$\mathcal{E}^{\sharp}_{\ell_{1},\dots,\ell_{n}}$ to $Z(\mathcal{L})$. Then the equivalence relations
$\mathcal{E}_{\ell_{1},\dots,\ell_{n}}$ generate the uniformity on $Z(\mathcal{L})$. In particular,
$Z(\mathcal{L})$ is a $|A|^{+}$-totally bounded non-Archimedean uniform space.

Let $(X,\mathcal{U})$ be a uniform space. Then let $\mathfrak{B}_{A}(X,\mathcal{U})$ be the
collection of all uniformly continuous mappings from $X$ to $A$. Clearly $\mathfrak{B}_{A}(X,\mathcal{U})$ is
a subdirect product of $\Omega(A)$, so $\mathfrak{B}_{A}(X,\mathcal{U})$ is a partitionable algebra.

If $(X,\mathcal{U})$ is a uniform space, then for each $x\in X$, we have
$\pi_{x}\colon\mathfrak{B}_{A}(X,\mathcal{U})\to\Omega(A)$ be a homomorphism where $\pi_{x}$ is
the projection mapping defined by $\pi_{x}(f)=f(x)$.
Therefore define a mapping $\mathcal{C}\colon(X,\mathcal{U})\to
Z(\mathfrak{B}_{A}(X,\mathcal{U}))$ by $\mathcal{C}(x)=\pi_{x}$. In other words, if
$x\in X$, and $f\colon(X,\mathcal{U})\to A$ is uniformly continuous, then
$\mathcal{C}(x)f=f(x)$. If there is any confusion about the space $(X,\mathcal{U})$, then we shall write
$\mathcal{C}_{(X,\mathcal{U})}$ for the mapping $\mathcal{C}$.

Now let $\mathcal{L}\in V(\Omega(A))$. If $\ell\in\mathcal{L}$, then let $\ell^{\star}\colon Z(\mathcal{L})\to A$
be the mapping defined by $\ell^{\star}(\phi)=\phi(\ell)$. We claim that $\ell^{\star}$ is
uniformly continous. Assume that $(\phi,\theta)\in\mathcal{E}_{\ell}$. Then
$\phi(\ell)=\theta(\ell)$, so $\ell^{\star}(\phi)=\ell^{\star}(\theta)$, and hence
$(\ell^{\star}(\phi),\ell^{\star}(\theta))\in E$ for each equivalence relation $E$ on $A$.
Therefore $\ell^{\star}$ is uniformly continuous, so $\ell^{\star}\in\mathfrak{B}_{A}(Z(\mathcal{L}))$.
In light of the above discussion, we define a function $\rho\colon\mathcal{L}\to\mathfrak{B}_{A}(Z(\mathcal{L}))$ by
$\rho(\ell)=\ell^{\star}$. Therefore $\rho(\ell)(\phi)=\phi(\ell)$ for $\phi\in Z(\mathcal{L}),\ell\in\mathcal{L}$.
We will write $\rho_{\mathcal{L}}$ for the mapping $\rho$ to specify the domain of $\rho$ in case there may be confusion.
\begin{ex}
If $f\colon A^{n}\to A$ is injective (surjective), then $\hat{f}^{\mathcal{L}}\colon\mathcal{L}^{n}\to\mathcal{L}$
is injective (surjective) for each $\mathcal{L}\in V(\Omega(A))$.
\end{ex}
\begin{theorem}
The equivalence relations $\mathcal{E}_{\ell}$ generate the uniformity on $Z(\mathcal{L})$.
\end{theorem}
\begin{proof}
Assume that $\ell_{1},\dots,\ell_{n}\in\mathcal{L}$. Let $i\colon A^{n}\to A$ be injective.
Then $\hat{i}^{\mathcal{L}}$ is also injective. Now let $\ell=\hat{i}^{\mathcal{L}}(\ell_{1},\dots,\ell_{n})$.
Assume $\phi,\theta\in Z(\mathcal{L})$ and $(\phi,\theta)\in\mathcal{E}_{\ell}$. Then
$\phi(\ell)=\theta(\ell)$, so $\phi(\hat{i}^{\mathcal{L}}(\ell_{1},\dots,\ell_{n}))=
\theta(\hat{i}^{\mathcal{L}}(\ell_{1},\dots,\ell_{n}))$. Therefore,
$i(\phi(\ell_{1}),\dots,\phi(\ell_{n}))=i(\theta(\ell_{1}),\dots,\theta(\ell_{n}))$, so since
$i$ is injective, we have $\phi(\ell_{1})=\theta(\ell_{1}),\dots,\phi(\ell_{n})=\theta(\ell_{n})$,
thus $(\phi,\theta)\in\mathcal{E}_{\ell_{1},\dots,\ell_{n}}$. In other words, we have 
$\mathcal{E}_{\ell}\subseteq\mathcal{E}_{\ell_{1},\dots,\ell_{n}}$. Therefore the equivalence relations
$\mathcal{E}_{\ell}$ generate the uniformity on $Z(\mathcal{L})$.
\end{proof}

\begin{theorem}
1. Let $\mathcal{L}\in V(\Omega(A))$. Then $\rho\colon\mathcal{L}\to
\mathfrak{B}_{A}(Z(\mathcal{L}))$ is a surjective homomorphism, and $\rho$ is an isomorphism
if and only if $\mathcal{L}$ is partitionable.

2. If $(X,\mathcal{U})$ is a uniform space, then the mapping
$\mathcal{C}\colon(X,\mathcal{U})\to Z(\mathfrak{B}_{A}(X,\mathcal{U}))$ is uniformly continuous
and $\mathcal{C}''(X)$ is dense in $Z(\mathfrak{B}_{A}(X,\mathcal{U}))$. If $(X,\mathcal{U})$ is separated and
non-Archimedean, then $\mathcal{C}$ is injective. If $(X,\mathcal{U})$
is separated non-Archimedean and $|A|^{+}$-totally bounded, then $\mathcal{C}$ is an embedding.
If $(X,\mathcal{U})$ is complete non-Archimedean and $|A|^{+}$-totally bounded, then
$\mathcal{C}$ is an isomorphism.
\end{theorem}
\begin{proof}
1. If $\ell\in\mathcal{L}$, then we have $\rho(\ell)=(\rho(\ell)(\phi))_{\phi\in Z(\mathcal{L})}=(\phi(\ell))_{\phi\in Z(\mathcal{L})}$.
Therefore $\rho$ is a homomorphism since $\rho$ is a homomorphism in each coordinate.

To prove surjectivity, assume that $f\colon Z(\mathcal{L})\to A$ is uniformly continuous.
Then there is an $\ell\in\mathcal{L}$ where if $(\phi,\theta)\in\mathcal{E}_{\ell}$, then
$f(\phi)=f(\theta)$. In other words, if $\phi(\ell)=\theta(\ell)$, then $f(\phi)=f(\theta)$.
Therefore there is a function $g\colon A\to A$ where $f(\phi)=g(\phi(\ell))$ whenever $\phi\in Z(\mathcal{L})$.
Furthermore, we have $f(\phi)=g(\phi(\ell))=\phi(\hat{g}^{\mathcal{L}}(\ell))=\rho(\hat{g}^{\mathcal{L}}(\ell))(\phi)$ for each
$\phi\in Z(\mathcal{L})$. Therefore $\rho(\hat{g}^{\mathcal{L}}(\ell))=f$. Thus the mapping $\rho$ is surjective.

Now assume $\mathcal{L}$ is partitionable. Then for each pair of distinct $\ell_{1},\ell_{2}\in\mathcal{L}$ there is
a homomorphism $\phi\colon\mathcal{L}\to A$ with $\rho(\ell_{1})(\phi)=\phi(\ell_{1})\neq\phi(\ell_{2})=\rho(\ell_{2})(\phi)$.
Therefore $\rho(\ell_{1})\neq\rho(\ell_{2})$. We conclude that $\rho$ is injective. 
Likewise, if we assume $\rho$ is an isomorphism, then since $\mathfrak{B}_{A}(Z(\mathcal{L}))$ is partitionable, we have
$\mathcal{L}$ be partitionable as well.

2. Since $\mathcal{C}\colon(X,\mathcal{U})\to Z(\mathfrak{B}_{A}(X,\mathcal{U}))\subseteq A^{\mathfrak{B}_{A}(X,\mathcal{U})}$, we have
$\mathcal{C}$ be uniformly continuous if and only if $\mathcal{C}$ is uniformly continuous in every coordinate $f\in\mathfrak{B}_{A}(X,\mathcal{U})$.
However, we have $\mathcal{C}(x)=(\mathcal{C}(x)(f))_{f\in\mathfrak{B}_{A}(X,\mathcal{U})}=(f(x))_{f\in\mathfrak{B}_{A}(X,\mathcal{U})}$, so
$\mathcal{C}$ is uniformly continuous.

We shall now show that $\mathcal{C}''(X)$ is dense in $Z(\mathfrak{B}_{A}(X,\mathcal{U}))$.
The uniformity on $Z(\mathfrak{B}_{A}(X,\mathcal{U}))$ is generated by the
equivalence relations $\mathcal{E}_{f}$ where $f\in\mathfrak{B}_{A}(X,\mathcal{U})$.
The blocks in the equivalence relation $\mathcal{E}_{f}$ are the nonempty sets of the form
$U_{f,a}=\{\phi\in Z(\mathfrak{B}_{A}(X,\mathcal{U}))|\phi(f)=a\}$. Therefore it suffices to show
that $\mathcal{C}''(X)$ intersects each non-empty block $U_{f,a}$.

Now assume that $U_{f,a}$ is non-empty. Then there is a $\phi\in Z(\mathfrak{B}_{A}(X,\mathcal{U}))$ with
$\phi(f)=a$. We claim that $f(x)=a$ for some $x\in X$. Therefore, assume that $f(x)\neq a$ for all
$x\in X$. Let $i\colon A\to A$ be a mapping where $i(a)\neq a$ and $i(b)=b$ for $b\neq a$.
Then we have $f=i\circ f=\hat{i}^{\mathfrak{B}_{A}(X,\mathcal{U})}(f)$, so
$\phi(f)=\phi(\hat{i}^{\mathfrak{B}_{A}(X,\mathcal{U})}(f))=i(\phi(f))\neq a$. Thus, by
contrapositive, if $\phi(f)=a$, then $f(x)=a$ for some $x\in X$. However, we have
$\mathcal{C}(x)(f)=f(x)=a$, so $\mathcal{C}(x)\in U_{f,a}$. Therefore $\mathcal{C}''(X)$ is
dense in $Z(\mathfrak{B}_{A}(X,\mathcal{U}))$.

Now assume that $(X,\mathcal{U})$ is separated and non-Archimedean.
Then we shall show that $\mathcal{C}$ is injective.
Assume that $x,y\in X,x\neq y$. Then since $(X,\mathcal{U})$
is separated and non-Archimedean, there is a uniformly continuous function $f\colon X\to A$
such that $f(x)\neq f(y)$. Therefore $\mathcal{C}(x)(f)=f(x)\neq f(y)=\mathcal{C}(y)(f)$, and hence
 $\mathcal{C}(x)\neq\mathcal{C}(y)$. We conclude that $\mathcal{C}$ is injective.

Now assume that $(X,\mathcal{U})$ is separated, non-Archimedean, and $|A|^{+}$-totally bounded.
Then we shall show that $\mathcal{C}$ is an embedding. Assume that $E\in\mathcal{U}$ is an equivalence relation.
Then since $(X,\mathcal{U})$ is $|A|^{+}$-totally bounded, there is a function $f\colon X\to A$ where
$f(x)=f(y)$ if and only if $(x,y)\in E$. Clearly $f$ is uniformly continuous, so
$f\in\mathfrak{B}_{A}(X,\mathcal{U})$ and $\mathcal{E}_{f}$ is an equivalence relation on $Z(\mathfrak{B}_{A}(X,\mathcal{U}))$. Now assume that $x,y\in X$. Then $(x,y)\in E$ if and only if
$f(x)=f(y)$ if and only if $\mathcal{C}(x)(f)=\mathcal{C}(y)(f)$ if and only if $(\mathcal{C}(x),\mathcal{C}(y))\in\mathcal{E}_{f}$.
Therefore $\mathcal{C}$ is an embedding.

If $(X,\mathcal{U})$ is complete, non-Archimedean, and $|A|^{+}$-totally bounded, then we have
$\mathcal{C}$ be an embedding, and $Z(\mathfrak{B}_{A}(X,\mathcal{U}))$ is the completion of
$\mathcal{C}''(X)$. However, if $X$ is complete, we have $\mathcal{C}''(X)=
Z(\mathfrak{B}_{A}(X,\mathcal{U}))$. Therefore, in this case, $\mathcal{C}$ is a uniform homeomorphism.
\end{proof}

Let $\mathcal{L},\mathcal{M}\in V(\Omega(A))$ and assume that $\phi\colon\mathcal{L}\to\mathcal{M}$
is a homomorphism. Then let $Z(\phi)\colon Z(\mathcal{M})\to Z(\mathcal{L})$ be the function defined by
$Z(\phi)(\theta)=\theta\circ\phi$ for homomorphisms $\theta\colon\mathcal{M}\to A$. One can easily show that
the mappings $Z(\phi)$ are uniformly continuous and $Z$ is a functor from the variety $V(\Omega(A))$ to
the category of uniform spaces. Now assume that $(X,\mathcal{U}),(Y,\mathcal{V})$ are uniform spaces and
$f\colon(X,\mathcal{U})\to(Y,\mathcal{V})$ is uniformly continuous. Then define a mapping
$\mathfrak{B}_{A}(f)\colon\mathfrak{B}_{A}(Y,\mathcal{V})\to\mathfrak{B}_{A}(X,\mathcal{U})$
by $\mathfrak{B}_{A}(f)(g)=g\circ f$. Then each $\mathfrak{B}_{A}(f)$ is a homomorphism. Furthermore,
$\mathfrak{B}_{A}$ gives a functor from the category of uniform spaces to the variety $V(\Omega(A))$.

\begin{theorem}
\begin{enumerate}
\item Let $f\colon(X,\mathcal{U})\to(Y,\mathcal{V})$ be uniformly continuous. Then
$Z(\mathfrak{B}_{A}(f))\circ\mathcal{C}_{(X,\mathcal{U})}=\mathcal{C}_{(Y,\mathcal{V})}\circ f$.

$$\begin{CD}
(X,\mathcal{U})    		 		@>f>>    															(Y,\mathcal{V})\\
@VV\mathcal{C}V														@VV\mathcal{C}V\\
Z(\mathfrak{B}_{A}(X,\mathcal{U})) @>Z(\mathfrak{B}_{A}(f))>>   Z(\mathfrak{B}_{A}(Y,\mathcal{V}))
\end{CD}$$

\item Let $\phi\colon\mathcal{L}\to\mathcal{M}$ be a homomorphism. Then we have
$\mathfrak{B}_{A}(Z(\phi))\circ\rho_{\mathcal{L}}=\rho_{\mathcal{M}}\circ\phi$.

$$\begin{CD}
\mathcal{L}    								@>\phi>>    															\mathcal{M}\\
@VV\rho V														    @VV\rho V\\
\mathfrak{B}_{A}(Z(\mathcal{L}))	 @>\mathfrak{B}_{A}(Z(\phi))>>   \mathfrak{B}_{A}(Z(\mathcal{M}))
\end{CD}$$

\item The pair of functions $Z(\rho_{\mathcal{L}})\colon Z(\mathfrak{B}_{A}(Z(\mathcal{L})))\to Z(\mathcal{L})$
and $\mathcal{C}_{Z(\mathcal{L})}\colon Z(\mathcal{L})\to Z(\mathfrak{B}_{A}(Z(\mathcal{L})))$ are
inverses.

\item The pair of functions $\mathfrak{B}_{A}(\mathcal{C}_{(X,\mathcal{U})})\colon\mathfrak{B}_{A}(Z(\mathfrak{B}_{A}(X,\mathcal{U})))
\to\mathfrak{B}_{A}(X,\mathcal{U})$ and $\rho_{\mathfrak{B}_{A}(X,\mathcal{U})}
\colon\mathfrak{B}_{A}(X,\mathcal{U})\to\mathfrak{B}_{A}(Z(\mathfrak{B}_{A}(X,\mathcal{U})))$ are
inverses.

\end{enumerate}
\end{theorem}
\begin{proof}
\begin{enumerate}
\item Let $x\in X$ and let $g\in\mathfrak{B}_{A}(Y,\mathcal{V})$. Then we have
$$[(Z(\mathfrak{B}_{A}(f))\circ\mathcal{C})(x)](g)=[Z(\mathfrak{B}_{A}(f))(\mathcal{C}(x))](g)$$
$$=[\mathcal{C}(x)\circ\mathfrak{B}_{A}(f)]g=\mathcal{C}(x)[\mathfrak{B}_{A}(f)(g)]$$
$$=\mathcal{C}(x)(g\circ f)=g(f(x))=\mathcal{C}(f(x))(g).$$
Therefore $\mathcal{C}\circ f=Z(\mathfrak{B}_{A}(f))\circ\mathcal{C}$.

\item This proof is analogous to part 1. Let $\ell\in\mathcal{L}$ and let $\theta\in Z(\mathcal{M})$. Then we have
$$[(\mathfrak{B}_{A}(Z(\phi))\circ\rho)(\ell)](\theta)=[\mathfrak{B}_{A}(Z(\phi))(\rho(\ell))](\theta)$$
$$=[\rho(\ell)\circ Z(\phi)]\theta=\rho(\ell)(Z(\phi)(\theta))$$
$$=\rho(\ell)(\theta\circ\phi)=\theta\circ\phi(\ell)=\theta(\phi(\ell))=\rho(\phi(\ell))(\theta).$$

Therefore $\rho\circ\phi=\mathfrak{B}_{A}(Z(\phi))\circ\rho$.

\item The uniform space $Z(\mathcal{L})$ is complete, so $\mathcal{C}_{Z(\mathcal{L})}$ is a uniform
homeomorphism. It therefore suffices to show that $Z(\rho_{\mathcal{L}})\circ\mathcal{C}_{Z(\mathcal{L})}\colon
Z(\mathcal{L})\to Z(\mathcal{L})$
is the identity map. Therefore let $\phi\colon\mathcal{L}\to\Omega(A)$ is a homomorphism and $\ell\in\mathcal{L}$.
Then we have
$$[Z(\rho_{\mathcal{L}})\circ\mathcal{C}_{Z(\mathcal{L})}(\phi)](\ell)
=[Z(\rho_{\mathcal{L}})(\mathcal{C}_{Z(\mathcal{L})}(\phi))](\ell)$$
$$=[\mathcal{C}_{Z(\mathcal{L})}(\phi)\circ\rho_{\mathcal{L}}](\ell)=
\mathcal{C}_{Z(\mathcal{L})}(\phi)(\rho_{\mathcal{L}}(\ell))=\rho_{\mathcal{L}}(\ell)(\phi)=\phi(\ell).$$
We therefore conclude that $Z(\rho_{\mathcal{L}})\circ\mathcal{C}_{Z(\mathcal{L})}$ is the identity map.

\item This proof this analogous to 3. Since $\mathfrak{B}_{A}(X,\mathcal{U})$ is partitionable, we have $\rho_{\mathfrak{B}_{A}(X,\mathcal{U})}$ be an
isomorphism. We therefore need to show that $\mathfrak{B}_{A}(\mathcal{C}_{(X,\mathcal{U})})\circ\rho_{\mathfrak{B}_{A}(X,\mathcal{U})}\colon
\mathfrak{B}_{A}(X,\mathcal{U})\to\mathfrak{B}_{A}(X,\mathcal{U})$
is the identity map. Thus, assume that $f\in\mathfrak{B}_{A}(X,\mathcal{U})$ and $x\in X$. Then
$$[\mathfrak{B}_{A}(\mathcal{C}_{(X,\mathcal{U})})\circ\rho_{\mathfrak{B}_{A}(X,\mathcal{U})}(f)](x)
=[\mathfrak{B}_{A}(\mathcal{C}_{(X,\mathcal{U})})(\rho_{\mathfrak{B}_{A}(X,\mathcal{U})}(f))](x)$$
$$=(\rho_{\mathfrak{B}_{A}(X,\mathcal{U})}(f)\circ\mathcal{C}_{(X,\mathcal{U})})(x)
=\rho_{\mathfrak{B}_{A}(X,\mathcal{U})}(f)(\mathcal{C}_{(X,\mathcal{U})}(x))$$
$$=\mathcal{C}_{(X,\mathcal{U})}(x)(f)=f(x).$$

Therefore $\mathfrak{B}_{A}(\mathcal{C}_{(X,\mathcal{U})})\circ\phi_{\mathfrak{B}_{A}(X,\mathcal{U})}$ is the identity map.
\end{enumerate}
\end{proof}

\section{A Characterization of non-Archimedean Supercomplete Spaces}

A congruence $\theta$ on $\mathcal{L}$ is said to be \emph{partitionable} if
$\mathcal{L}/\theta$ is partitionable. 
Let $PC(\mathcal{L})$ denote the collection of all partitional congruences of $\mathcal{L}$. One can easily
see that $PC(\mathcal{L})$ consists of all congruences of the form $\bigcap_{\theta\in R}\ker(\theta)$ where
$R\subseteq Z(\mathcal{L})$.

\begin{theorem}
Let $\mathcal{L}\in V(\Omega(A))$. Let $R\subseteq Z(\mathcal{L})$. Then let
$\phi\in Z(\mathcal{L})$. Then $\phi\in\overline{R}$ if and only if $\bigcap_{\theta\in R}\ker(\theta)\subseteq\ker(\phi)$.
\end{theorem}
\begin{proof}
$\rightarrow$ Assume $\phi\in\overline{R}$. Also assume $\ell,\mathfrak{m}\in\mathcal{L}$ and $(\ell,\mathfrak{m})\in\bigcap_{\theta\in R}\ker(\theta)$. Then $\theta(\ell)=\theta(\mathfrak{m})$ for $\theta\in R$. Since $\phi\in\overline{R}$, there is a
$\theta\in R$ with $(\phi,\theta)\in\mathcal{E}_{\ell,\mathfrak{m}}$, so
$\phi(\ell)=\theta(\ell)=\theta(\mathfrak{m})=\phi(\mathfrak{m})$. Therefore $(\ell,\mathfrak{m})\in\ker(\phi)$.
We conclude that $\bigcap_{\theta\in R}\ker(\theta)\subseteq\ker(\phi)$.

$\leftarrow$ Assume $\bigcap_{\theta\in R}\ker(\theta)\subseteq\ker(\phi)$. Then let $\ell\in\mathcal{L}$ and
assume $\phi(\ell)=a$. Let $b\in A$ be an element with $b\neq a$. Let $i\colon A\to A$ be the map where
$i(a)=a$ and $i(c)=b$ for $c\neq a$. Then $\phi(\hat{i}^{\mathcal{L}}(\ell))=i(\phi(\ell))=i(a)=a\neq b=\phi(\hat{b}^{\mathcal{L}})$,
so $(\hat{i}^{\mathcal{L}}(\ell),\hat{b}^{\mathcal{L}})\not\in\ker(\phi)$, hence
$(\hat{i}^{\mathcal{L}}(\ell),\hat{b}^{\mathcal{L}})\not\in\ker(\theta)$ for some $\theta\in R$. Therefore
$b=\theta(\hat{b}^{\mathcal{L}})\neq\theta(\hat{i}^{\mathcal{L}}(\ell))=i(\theta(\ell))$. Thus
$\theta(\ell)=a=\phi(\ell)$. Therefore $(\phi,\theta)\in\mathcal{E}_{\ell}$. Since $\ell\in\mathcal{L}$ is
arbitrary, we have $\phi\in\overline{R}$.
\end{proof}
We shall now give a Galois correspondence between closed sets in $Z(\mathcal{L})$ and
partitionable congruences in $\mathcal{L}$. Let
$f\colon P(\mathcal{L}^{2})\to P(Z(\mathcal{L})),g\colon P(Z(\mathcal{L}))\to P(\mathcal{L}^{2})$
be the mappings where \[f(R)=\{\phi\in Z(\mathcal{L})|(a,b)\in\ker(\phi)\,\textrm{for all}\,(a,b)\in R\}
=\{\phi\in Z(\mathcal{L})|R\subseteq\ker(\phi)\}\]
and where \[g(S)=\{(a,b)\in\mathcal{L}^{2}|(a,b)\in\ker(\phi)\,\textrm{for all}\,\phi\in S\}
=\bigcap_{\phi\in S}\ker(\phi).\]

Let $C=g\circ f,D=f\circ g$. Then $C$ and $D$ are closure operators. In other words, we have
$C(R)\subseteq C(C(R))$, and if $R\subseteq S$, then $C(R)\subseteq C(S)$ for $R,S\subseteq\mathcal{L}^{2}$.
Let $C^{*}=\{R\subseteq\mathcal{L}^{2}|C(R)=R\}=\{C(R)|R\subseteq\mathcal{L}^{2}\}$ and let
$D^{*}=\{S\subseteq Z(\mathcal{L})|D(S)=D\}=\{D(S)|S\subseteq Z(\mathcal{L})\}$.
Let $f^{*}\colon C^{*}\to D^{*},g^{*}\colon D^{*}\to C^{*}$ be the restriction of the functions
$f$ and $g$. Then the functions $f^{*}$ and $g^{*}$ are inverse functions.

\begin{theorem}
The mapping $D$ is the topological closure operator induced by the uniformity on $Z(\mathcal{L})$.
\end{theorem}
\begin{proof}
Let $R\subseteq Z(\mathcal{L})$. Then \[D(R)=f\circ g(R)
=f(\bigcap_{\theta\in R}\ker(\theta))=\{\phi\in Z(\mathcal{L})|\bigcap_{\theta\in R}\ker(\theta)\subseteq\ker(\phi)\}
=\overline{R}.\]
\end{proof}

If $(X,\mathcal{U})$ is a uniform space, then let $H(X)$ be the collection
of all closed subsets of $X$. Clearly $D^{*}=H(Z(\mathcal{L}))$ and $C^{*}=PC(\mathcal{L})$.
Therefore we have $f^{*}\colon PC(\mathcal{L})\to H(Z(\mathcal{L}))$ and
$g^{*}\colon H(Z(\mathcal{L}))\to PC(\mathcal{L})$. 

We shall now characterize the partitionable algebras $\mathcal{L}$ where $S(\mathcal{L})$ is supercomplete.
For each $E\in\mathcal{U}$, let $\overline{E}$
be the binary relation on $H(X)$ where $(C,D)\in\overline{E}$ if and only if
$C\subseteq E[D]=\{x\in X|(z,x)\in E\,\textrm{for some}\,z\in D\}$ and $D\subseteq E[C]$. Then the relations $\overline{E}$ generate
a uniformity on $H(X)$. Therefore $H(X)$ is a uniform space. With this uniformity, we shall call
$H(X)$ the hyperspace of $X$. A separated uniform space $X$ is said to be \emph{supercomplete} if $H(X)$ is complete.

Take note that if $\mathcal{L}$ is an algebra and $\ell\in\mathcal{L}$, then we have
$\phi\in\mathcal{E}_{\ell}[C]$ if and only if there is some $\theta\in C$ with
$(\theta,\phi)\in\mathcal{E}_{\ell}$. In other words, $\phi\in\mathcal{E}_{\ell}[C]$ if and only if
$\phi(\ell)\in\{\theta(\ell)|\theta\in C\}$. Therefore
$(C,D)\in\overline{\mathcal{E}_{\ell}}$ if and only if $\{\theta(\ell)|\theta\in C\}=\{\phi(\ell)|\phi\in D\}$.

\begin{ex}
Every finitely generated algebra $\mathcal{L}\in V(\Omega(A))$ is generated by a single element.
\end{ex}

A \emph{locally partitionable} congruence is a congruence $\theta$ on $\mathcal{L}$ so that
whenever $\mathcal{M}\subseteq\mathcal{L}$ is a finitely generated subalgebra, we have
$\theta\cap\mathcal{M}^{2}$ be a partitionable congruence. 

Let $LPC(\mathcal{L})$ denote the set of all locally partitionable congruences on $\mathcal{L}$.
$LPC(\mathcal{L})$ is closed under arbitrary intersection, so $LPC(\mathcal{L})$ is a complete lattice.
Let $FS(\mathcal{L})$ be the collection of all finitely generated subalgebras of $\mathcal{L}$. 
We shall now give $LPC(\mathcal{L})$ a complete uniformity by representing $LPC(\mathcal{L})$ as an inverse limit.

If $\mathcal{M},\mathcal{N}$ are finitely generated subalgebras of $\mathcal{L}$ and
$\mathcal{M}\subseteq\mathcal{N}$, then define a function 
$E_{\mathcal{N},\mathcal{M}}\colon PC(\mathcal{N})\to PC(\mathcal{M})$ by letting
$E_{\mathcal{N},\mathcal{M}}(\theta)=\theta\cap\mathcal{M}^{2}$. One can easily show that
$(PC(\mathcal{N}))_{\mathcal{N}\in FS(\mathcal{L})}$ is an inverse system of sets with transitional
mappings $E_{\mathcal{N},\mathcal{M}}$. Let $IL(\mathcal{L})$ be the inverse limit $^{Lim}_{\longleftarrow}PC(\mathcal{N})$. Give each
$PC(\mathcal{N})$ the discrete uniformity and give $^{Lim}_{\longleftarrow}PC(\mathcal{N})$ the inverse limit uniformity.
Let $\mathcal{E}_{\mathcal{N}}$ be the equivalence relation on $IL(\mathcal{L})$ where we have
$(\theta_{\mathcal{M}})_{\mathcal{M}\in FS(\mathcal{L})},(\psi_{\mathcal{M}})_{\mathcal{M}\in FS(\mathcal{L})}\in\mathcal{E}_{\mathcal{N}}$ if and only if $\theta_{\mathcal{N}}=\psi_{\mathcal{N}}$. Then the equivalence relations $\mathcal{E}_{\mathcal{N}}$ generate
the uniformity on $IL(\mathcal{L})$.

Let $\Gamma\colon LPC(\mathcal{L})\to IL(\mathcal{L})$ be the mapping defined by letting
$\Gamma(\theta)=(\theta\cap\mathcal{M}^{2})_{\mathcal{M}\in FS(\mathcal{L})}$. Conversely, define a mapping
$\Delta\colon IL(\mathcal{L})\to LPC(\mathcal{L})$ be the mapping defined by
$\Delta((\theta_{\mathcal{M}})_{\mathcal{M}\in FS(\mathcal{L})})=\bigcup_{\mathcal{M}}\theta_{\mathcal{M}}$.

\begin{ex}
The functions $\Gamma$ and $\Delta$ are inverses.
\end{ex}
Now give $LPC(\mathcal{L})$ the uniformity such that the maps $\Gamma$ and $\Delta$ are uniform
homeomorphisms. Now for each finitely generated subalgebra $\mathcal{N}\subseteq\mathcal{L}$, let
$\mathcal{F}_{\mathcal{N}}$ be the equivalence relation on $LPC(\mathcal{L})$ where
$(\theta,\psi)\in\mathcal{F}_{\mathcal{N}}$ if and only if $\theta\cap\mathcal{N}^{2}=\psi\cap\mathcal{N}^{2}$.
Clearly $(\theta,\psi)\in\mathcal{F}_{\mathcal{N}}$ if and only if
$(\Gamma(\theta),\Gamma(\psi))\in\mathcal{E}_{\mathcal{N}}$. Therefore the equivalence relations
$\mathcal{F}_{\mathcal{N}}$ generate the uniformity on $LPC(\mathcal{L})$.

\begin{ex}
Let $(X,\mathcal{U})$ be a non-Archimedean uniform space.
Let $\mathcal{N}\subseteq\mathfrak{B}_{A}(X,\mathcal{U})$ be a finitely generated subalgebra.
Then there is a partition $P$ such that if $r\colon X\to P$ is the function where $x\in r(x)$ for
all $x\in X$, then $\mathcal{N}=\{f\circ r|f\colon P\to A\}$. Furthermore, if $\theta$ is a partitionable
congruence on $\mathcal{N}$, then there is an $V\subseteq X$ where if $f,g\in\mathcal{N}$, then $(f,g)\in\theta$ if and only if $f(x)=g(x)$
for all $x\in V$.
\end{ex}

\begin{theorem}
Let $\mathcal{L}$ be partitionable. Then $PC(\mathcal{L})$ is dense in $LPC(\mathcal{L})$.
\end{theorem}
\begin{proof}
Since $\mathcal{L}$ is partitionable, we may assume that $\mathcal{L}=\mathfrak{B}_{A}(X,\mathcal{U})$ for
some complete non-Archimedean $|A|^{+}$-totally bounded uniform space $(X,\mathcal{U})$. Let
$\theta\in LPC(\mathcal{L})$ and assume that $\mathcal{N}\subseteq\mathfrak{B}_{A}(X,\mathcal{U})$ is finitely
generated. Then there is a $V\subseteq X$ where for $f,g\in\mathcal{N}$, we have $(f,g)\in\theta$ if and only if
$f(x)=g(x)$ for all $x\in V$. Now let $V^{\sharp}$ be the congruence in $\mathfrak{B}_{A}(X,\mathcal{U})$
where $(f,g)\in V^{\sharp}$ if and only if $f(x)=g(x)$ for $x\in V$. Then $V^{\sharp}$ is a partitionable
congruence with $V^{\sharp}\cap\mathcal{N}^{2}=\theta\cap\mathcal{N}^{2}$. Therefore
$(V^{\sharp},\theta)\in\mathcal{F}_{\mathcal{N}}$. We conclude that $PC(\mathcal{L})$ is dense in $LPC(\mathcal{L})$.
\end{proof}

\begin{ex}
Assume $a_{i}\in A$ for $i\in I$ and $b_{j}\in A$ for $j\in J$. Then $\{a_{i}|i\in I\}=\{b_{j}|j\in J\}$ if and only if
for each pair of functions $f,g\colon A\to A$, we have $\forall i\in I,f(a_{i})=g(a_{i})\Leftrightarrow
\forall j\in J,f(b_{j})=g(b_{j})$.
\end{ex}
\begin{theorem}
The mappings $f^{*}\colon PC(\mathcal{L})\to H(Z(\mathcal{L}))$ and $g^{*}\colon H(Z(\mathcal{L}))\to PC(\mathcal{L})$
are uniform homeomorphisms.
\end{theorem}
\begin{proof}
We only need to show that $g^{*}$ is a uniform homeomorphism. Since $Z(\mathcal{L})$ is generated by equivalence relations
$\mathcal{E}_{\ell}$, the equivalence relations $\overline{\mathcal{E}_{\ell}}$ generate $H(Z(\mathcal{L}))$. We have
$(C,D)\in\overline{\mathcal{E}_{\ell}}$ if and only if
\[\{\theta(\ell)|\theta\in C\}=\{\theta(\ell)|\theta\in D\}\] if and only if for $f,g\colon A\to A$ we have
\[\forall\phi\in C,f(\phi(\ell))=g(\phi(\ell))\leftrightarrow\forall\phi\in D,f(\phi(\ell))=g(\phi(\ell))\]
if and only if for each $f,g\colon A\to A$ we have
\[\forall\phi\in C,\phi(\hat{f}^{\mathcal{L}}(\ell)))=\phi(\hat{g}^{\mathcal{L}}(\ell))\leftrightarrow
\forall\phi\in D,\phi(\hat{f}^{\mathcal{L}}(\ell)))=\phi(\hat{g}^{\mathcal{L}}(\ell))\]
if and only if whenever $f,g\colon A\to A$ we have
\[(\hat{f}^{\mathcal{L}}(\ell),\hat{g}^{\mathcal{L}}(\ell))\in\bigcap_{\phi\in C}\ker(\phi)\leftrightarrow
(\hat{f}^{\mathcal{L}}(\ell),\hat{g}^{\mathcal{L}}(\ell))\in\bigcap_{\phi\in D}\ker(\phi)\] if and only if
\[g^{*}(C)\cap\langle\ell\rangle^{2}=\bigcap_{\phi\in C}\ker(\phi)\cap\langle\ell\rangle^{2}=\bigcap_{\phi\in D}\ker(\phi)\cap\langle\ell\rangle^{2}=
g^{*}(D)\cap\langle\ell\rangle^{2}\] if and only if $(g^{*}(C),g^{*}(D))\in\mathcal{F}_{\langle\ell\rangle}$.
Therefore $g^{*}$ is a uniform homeomorphism.
\end{proof}
\begin{theorem}
Let $\mathcal{L}$ be a partitionable algebra. Then $Z(\mathcal{L})$ is supercomplete if and only if
every locally partitionable congruence on $\mathcal{L}$ is partitionable.
\end{theorem}
\begin{proof}
However, since $H(Z(\mathcal{L}))$ is uniformly
homeomorphic to $PC(\mathcal{L})$, we have $H(Z(\mathcal{L}))$ be complete if and only if $PC(\mathcal{L})$ is complete. Since
$PC(\mathcal{L})$ is a dense subspace of the complete space $LPC(\mathcal{L})$, we have $PC(\mathcal{L})$ be complete
if and only if $PC(\mathcal{L})=LPC(\mathcal{L})$ if and only if each locally partitionable congruence on $\mathcal{L}$ is
partitionable. Therefore $Z(\mathcal{L})$ is supercomplete if and only if every locally partitionable congruence
on $\mathcal{L}$ is partitionable.
\end{proof}

\begin{ex}
A partitionable algebra $\mathcal{L}$ is finitely generated if and only if
$Z(\mathcal{L})$ is discrete. A partitionable algebra $\mathcal{L}$ is countably
generated if and only if $Z(\mathcal{L})$ is uniformizable by a metric.
\end{ex}
We shall now prove a purely algebraic result using hyperspaces.
\begin{cor}
If $\mathcal{L}$ is a countably generated partitionable algebra, then every locally
partitionable congruence is partitionable.
\end{cor}
\begin{proof}
If $\mathcal{L}$ is a countably generated partitionable algebra, then $Z(\mathcal{L})$ is uniformizable
by a metric. However, in \cite{I}[p.\,\,30], it is shown that every complete metric space
is supercomplete. Therefore since $Z(\mathcal{L})$ is supercomplete, every locally
partitionable congruence is partitionable.
\end{proof}


\end{document}